\documentclass[11pt]{amsart}
\usepackage[left=2.2cm,tmargin=2.5cm,bmargin=2.5cm,right=2.2cm]{geometry}

\usepackage{physics,amsmath,amssymb,hyperref}

\theoremstyle{plain}
\newtheorem{thm}{Theorem}
\newtheorem{lemma}[thm]{Lemma}

\theoremstyle{definition}

\theoremstyle{remark}

\begin{document}

\title{Sets of unit fractions without two members whose average is a unit fraction}
\author{Will Sawin}

\maketitle

\begin{abstract} We show that there is a constant $c>0$ such that, for all sufficiently large $N$, there is a subset $A \subseteq \{1,\dots,N\}$ of size $>cN$ such that for any two distinct elements $a,b$ in $A$, the average of $\frac{1}{a}$ and $\frac{1}{b}$ is not a unit fraction, negatively answering a question of Erd\H{o}s and Graham. This also gives the best known lower bounds on the maximum size of a set of unit fractions without non-trivial  three-term arithmetic progressions.  \end{abstract}

Erd\H{o}s and Graham~\cite[p. 37]{erdos_graham_1980} asked whether, if $A \subseteq \{1,\dots, N\}$ is such that if $a,b\in A$ and $a\neq b$ then $a+b \nmid 2ab$, we must have $\abs{A} =o(N)$.  In this note, we give a negative answer to this question with the following explicit construction:

\begin{thm}\label{main} For a positive integer $N$, let $A_N$ be the set of $a \in \{1,\dots, N\}$ such that if $b \in \{1,\dots N\}$ satisfies $a\neq b$ and $\Omega(b) \leq \Omega(a)$ then $a+b \nmid 2ab$. Then

\begin{enumerate}

\item For $a,b\in A_N$ with $a\neq b$ we have $a+b \nmid 2ab$.

\item There exists a constant $c>0$ such that $\abs{A_N} > c N$ for all sufficiently large $N$.

\end{enumerate} \end{thm}

Here $\Omega(n)$ denotes the number of prime factors counted with multiplicity. 

Part (1) is immediate since of any pair $a,b$, we must have $\Omega(a) \leq \Omega(b)$ or $\Omega(b) \leq \Omega(a)$, so the only difficulty in the proof of Theorem \ref{main} is proving (2).

We have made no effort to compute the constant $c$ in Theorem \ref{main}, and hence have made no effort to optimize the construction. It seems near certain that the best possible lower bounds on the maximum size of $A$ will come from a different set $A$ defined by a more complicated condition. If the construction is optimized and an explicit lower bound is calculated, it would be interesting to compare with known upper bounds~\cite{LeonComment}.

The connection to unit fractions is that $a+b \mid 2ab$ if and only if $\frac{ \frac{1}{a} + \frac{1}{b}}{2}$ is a unit fraction.  So this problem concerns large sets of unit fractions without two distinct members whose average is a unit fraction. It follows immediately that $\{ \frac{1}{a} \mid a\in A\}$ is a set of unit fractions without non-trivial three-term arithmetic progressions, so our argument also gives a negative answer to a question recently raised by Korsky~\cite[Question 1.2]{Korsky} and improves on a construction from \cite{Korsky}.

This is a variant of the question, also asked by Erd\H{o}s and Graham~\cite[p. 37]{erdos_graham_1980}, of the largest size of a set $A \subseteq \{1,\dots, N\}$ such that if $a,b\in A$ with $a\neq b$ then $a+b \nmid ab$, or in other words, about large sets of unit fractions without two members whose sum is a unit fraction. For that problem, the set of all odd $a \in \{1,\dots, N\}$ produces an example of size $\lceil\frac{N}{2} \rceil$, and the main question is whether a substantially larger set exists. Our method also gives a lower bound for that problem, but worse than the bound arising from the set of odd numbers. It is possible that a sufficiently optimized version of our method could give a better bound and hence resolve that question as well.

The key idea of the proof is to restrict attention to a set $S$ of numbers $a$ which lack very small prime factors and do not have many more prime factors of a given size than expected. We count the number of $a \in S $  and compare to a bound for the number of $a$ with $a\in S$ but $a\notin A_N$. To bound the number of $a$ with $a\in S$ but $a\notin A_N$, it suffices to count pairs $a,b$ with $a\neq b$, $\Omega(b) \leq \Omega(a)$, and $a\in S$. This reduces by a change-of-variables to counting $a\in S$ which are divisible by certain integers $u(u+v)$, which can be done using known estimates for sums of nonnegative multiplicative functions. We use a result of de la Bret\'eche and Tenenbaum~\cite{LBT}, but it would likely be possible to instead use the earlier result \cite{Breteche2012} of the same authors, or the result of Matthiesen~\cite{Matthiesen2020}, which even gives an asymptotic for the relevant sums instead of simply an upper bound.

The crucial fact is that the $\Omega(b)\leq \Omega(a)$ and ``do not have many more prime factors of a given size than expected'' conditions make the average number of pairs $(a,b)$ for a given $a$ be bounded, when without these conditions it would, like the average value of Hooley's $\Delta$ function, be a power of $\log\log N$. The ``lacks very small prime factors'' condition lets us further reduce the average number of pairs $a,b$ for a given $a$ as small as we need.

The author must both acknowledge the use of AI in accordance with emerging professional standards and acknowledge works by human mathematicians which were inspirational to the author but which the proof does not logically depend on, and to do this, it is convenient to give a brief narrative of the source of the ideas: The story begins with a calculation of Stijn Cambie~\cite{CambieComment}, who found the largest set $A \subseteq \{1,\dots,500\}$ such that for $a,b\in A$ with $a\neq b$ we have $a+b \nmid 2ab$. The author asked ChatGPT to look for patterns in this set that could give a clue for how to generalize this construction, and it observed that for pairs $a,b$ with $a+b\mid 2ab$, the larger one is usually not in $A$, unless the smaller one is not in $A$ for other reasons, and also described a simple change-of-variables involving $u=\frac{a}{\gcd(a,b)}$, $v=\frac{b}{\gcd(a,b)}$. To get a construction which can be analyzed rigorously, it is natural to drop the ``usually'' and ``unless'' and simply consider the set of positive integers $a$ such that $a+b\nmid 2ab$ for all $b<a$. This turns out to be the set of $a$ which do not have two distinct divisors with ratio less than $2$ (i.e., more or less the set of numbers where the Hooley $\Delta$ function takes the value $1$). A lower bound for this set was found by Stef~\cite{StefThesis}, but it is not strong enough to give a negative answer to the question of Erd\H{o}s and Graham~\cite{erdos_graham_1980}. Examining the argument of \cite{StefThesis}, the author realized it would be more helpful to consider $b$ with $\Omega(b)\leq \Omega(a)$ instead of $b\leq a$. The strategy of proof then follows \cite{StefThesis}, suitably modified to apply to this problem. ChatGPT was also used for reference search and proofreading.

Specifically, the idea to restrict attention to the set $S$ is analogous to the strategy in \cite{StefThesis} to restrict attention to a certain set $U_{\alpha, T}$, and Lemma \ref{S-size} counting $S$ is analogous to \cite[Lemma 4.4]{StefThesis} counting $U_{\alpha,T}$. After restricting to $S$, we reduce in \eqref{bound-we-need} from counting $a$ divisible by $u(u+v)$ for certain pairs $u,(u+v)$ to counting the average number of pairs $u,v$ with $u(u+v)$ dividing $a$, which is analogous to \cite[p. 22, Demonstration, first displayed equation]{StefThesis}. After this, the argument needs to be different from \cite{StefThesis} to handle the crucial condition $\Omega(b)\leq \Omega(a)$, which becomes $\Omega(v) \leq \Omega(u)$.

The author was supported by NSF grant DMS-2502029 and was a Sloan Research Fellow while working on this manuscript.

\section{Proof}

Recall that $A_N$ is the set of $a \in \{1,\dots, N\}$ such that if $b \in \{1,\dots N\}$ satisfies $a\neq b$ and $\Omega(b) \leq \Omega(a)$ then $a+b \nmid 2ab$.

\begin{lemma}\label{divisibility-criterion} For positive integers $a,b$, we have $a+b\mid 2ab$ if and only if there are coprime positive integers $u,v$ with $u(u+v) \mid 2a$ and $b =av/u$. \end{lemma}

\begin{proof} Given $a,b$, set $u=a/\gcd(a,b)$ and $v= b/\gcd(a,b)$. Then certainly $u$ and $v$ are coprime positive integers and $b=av/u$.

 Then $a+b \mid 2ab$ if and only if $u+v \mid 2 uv \gcd(a,b)$ but $u$ and $v$ are coprime so this occurs if and only if $u+v \mid 2 \gcd(a,b)$ which implies $u+v \mid 2a$ and thus $u(u+v) \mid 2a$. 
 
 Conversely, if $u$ and $v$ are coprime and $b=av/u$ then $u=a/\gcd(a,b)$ and $v =b/\gcd(a,b)$. If $u (u+v)\mid 2a$ then $(u+v)\mid 2a$ and hence $(u+v) \mid 2b$ since $2b= 2av/u$ and $u$ is coprime to $u+v$. Thus $u+v \mid 2 \gcd(a,b)$ and therefore $a+b\mid 2ab$. \end{proof}

 \begin{lemma}\label{in-A-criterion} For $a\in \{1,\dots,N\}$, we have $a\in A_N$ if  $u(u+v) \nmid 2a $ for any pair $u,v$ of coprime positive integers with $v \leq uN/a$, $\Omega(v)\leq \Omega(u)$, and $(u,v)\neq (1,1)$.\end{lemma}
 
 Lemma \ref{in-A-criterion} can be made into an ``if and only if'' statement by adding the condition $u\mid a$. Since this condition is not helpful for our argument, we drop it. It would also be possible to drop the $u,v$ coprime condition, though this would make some later calculations messier.

\begin{proof} This follows immediately from Lemma \ref{divisibility-criterion}. Indeed, we suppose $a\notin A_N$, fix a witness $b$, and observe that, since $b=av/u$, that $b \leq N$ implies $v \leq u N/a$,  that $\Omega(b) \leq \Omega(a)$ implies $\Omega(v) \leq \Omega(u)$, and that $a\neq b$ implies $(u,v)\neq (1,1)$.\end{proof}

 In the remainder of the argument, we must show that there are many $a \in [1,N]$ with $2a$ not divisible by any $u(u+v)$ satisfying the conditions of Lemma \ref{in-A-criterion}. To do this, we will find a set of numbers $S$ such that the average over $a\in S$ of the number of $u,v$ pairs satisfying the conditions of Lemma \ref{in-A-criterion} with $u(u+v)\mid 2a$ is small. 
 
We fix parameters $\epsilon, \delta \in (0,1)$ and   $L>2$, and consider the set $S$ of natural numbers $a\in [\delta N, N]$, not divisible by any prime $<L$, with $\Omega(a,x) \leq (1+\epsilon )\log\log x$ for all $x \geq e$, where $\Omega(a,x)$ denotes the number of prime factors of $a$ that are less than or equal to $x$ counted with multiplicity.

\begin{lemma}\label{S-size} For $L$ sufficiently large with respect to $\epsilon$ and $\delta$ and $N$ sufficiently large, we have\[\abs{S} \geq(  \frac{1}{2}+o(1) )(1-\delta) N \prod_{p<L} (1-p^{-1})\] where $o(1)$ goes to $0$ as $N$ goes to $\infty$ with the other parameters fixed.
\end{lemma}

\begin{proof} The set of natural numbers $a \in [\delta N, N]$ not divisible by any prime $<L$ has size $(1+o(1)) (1-\delta) N \prod_{p<L} (1-p^{-1})$, so it suffices to show that the set of natural numbers $a\in [\delta N, N]$, not divisible by any prime $<L$, with $\Omega(a,x) > (1+\epsilon )\log\log x$ for some $x \geq e$ has size  $\leq (\frac{1}{2}+o(1)) (1-\delta) N \prod_{p<L} (1-p^{-1})$. 

We may assume $\epsilon<3$. For $\overline{\epsilon}=\epsilon/3$ so that $(1+\overline{\epsilon}^2<(1+\epsilon)$. Let $x_k= e^{e^{ (1+\overline{\epsilon})^k}}$. Then for any $x$, for $k$ minimal such that $x_k\geq x$ , we have $(1+\epsilon) \log \log x \geq (1+\overline{\epsilon})\log \log x_k)$ so that \[ \sum_{ \substack{ a\in [\delta N, N] \\ p \nmid a \textrm{ for }p < L \\ \Omega(a,x) > (1+\epsilon )\log\log x \textrm{ for some } x\geq e}}1  \leq \sum_{ \substack{ a\in [\delta N, N] \\ p \nmid a \textrm{ for }p < L \\ \Omega(a,x) > (1+\overline{\epsilon} )\log\log x_k \textrm{ for some } k\geq 0}}  1\leq  \sum_{k, x_k \geq L} \sum_{ \substack{ a\in [\delta N, N] \\ p \nmid a \textrm{ for }p < L \\ \Omega(a,x_k) > (1+\overline{\epsilon} )\log\log x_k }}  1\]
\[ \leq \sum_{k, x_k \geq L}  (1+\overline{\epsilon})^{-(1+\overline{\epsilon} )\log\log x_k } \sum_{ \substack{ a\in [\delta N, N] \\ p \nmid a \textrm{ for }p < L}} (1+\overline{\epsilon})^{\Omega(a,x_k)}. \]
The sum $\sum_{ \substack{ a\in [\delta N, N] \\ p \nmid a \textrm{ for }p < L}} (1+\overline{\epsilon})^{\Omega(a,x_k)}$ is the sum of a multiplicative function $f_k$ which is $ (1+\overline{\epsilon})^{\Omega(a,x_k)}$ if $p \nmid a$ for $p<L$ and $0$ otherwise. We have $f_k\geq 0$, $\sum_{p \leq y} f_k(p) \log p \ll y$, and $\sum_p \sum_{\nu \geq 2} \frac{ f_k(p^\nu) \log (p^\nu)}{p^\nu} \ll 1$, where the upper bounds follow from $f_k(a) \leq(1+\overline{\epsilon})^{\Omega(a)}$. Hence by \cite[Theorem III.3.5]{Tenenbaum2015} we have

\[\sum_{ \substack{ a\in [\delta N, N] \\ p \nmid a \textrm{ for }p < L}} (1+\overline{\epsilon})^{\Omega(a,x_k)}\ll  N \prod_{p \leq N} \begin{cases} 1- p^{-1} & \textrm{if } p<L \\ \frac{1}{1- (1+\overline{\epsilon}) p^{-1}} - p^{-1} & \textrm{if } L \leq p \leq x_k \\ \frac{1}{1-p^{-1}} - p^{-1}  & \textrm{if } p >x_k \end{cases}\]
\[ \ll N   (\log L)^{-1} \left( \frac{\log x_k}{\log L} \right)^{\overline{\epsilon}}  = N (\log x_k) ^{ \overline{\epsilon}} (\log L)^{-1-\overline{\epsilon}}\]
(absorbing a uniformly bounded Euler product into the implied constant) so that
\[\sum_{ \substack{ a\in [\delta N, N] \\ p \nmid a \textrm{ for }p < L \\ \Omega(a,x) > (1+\epsilon )\log\log x \textrm{ for some } x\geq e}}1 \]
\[ \ll \sum_{k, x_k \geq L} N (\log x_k)^{ \overline{\epsilon} - (1+\overline{\epsilon} ) \log (1+\overline{\epsilon})} (\log L)^{-1- \overline{\epsilon}}.\]

As $x_k=e^{e^{ (1+\overline{\epsilon})^k}}$ we have $\log x_k= e^{ (1+\overline{\epsilon})^k}$ which increases superexponentially with $k$. Since $\overline{\epsilon} - (1+\overline{\epsilon} ) \log (1+\overline{\epsilon}) <0$, the sum over $k$ is dominated by its first term, in which $\log x_k \geq \log L$, so we have

\[ \sum_{ \substack{ a\in [\delta N, N] \\ p \nmid a \textrm{ for }p < L \\ \Omega(a,x) > (1+\epsilon )\log\log x \textrm{ for some } x\geq e}}1  \ll  N  (\log L)^{-1 - (1+\overline{\epsilon} ) \log (1+\overline{\epsilon})} \]
\[\ll (\log L)^{- (1+\overline{\epsilon} ) \log (1+\overline{\epsilon})} (1-\delta) N \prod_{p<L} (1-p^{-1})\]
since $\prod_{p<L} (1-p^{-1})^{-1} \ll \log L$ and $(1-\delta)^{-1} \ll 1$. The quantity $(\log L)^{- (1+\overline{\epsilon} ) \log (1+\overline{\epsilon})} $ may be made arbitrarily small by taking $L$ sufficiently large, completing the proof. \end{proof}

Our goal will be to choose $\epsilon,\delta,L$ so that
\begin{equation}\label{bound-we-need}\sum_{\substack{u,v>0 \\ \textrm{coprime}\\ (u,v) \neq (1,1) \\ v \leq u \delta^{-1} \\ \Omega(v) \leq \Omega(u)}} \frac{ \abs{ \{a\in S \mid u(u+v) \mid 2a\}}}{ \abs{S}} <\frac{1}{2} \end{equation}
and thus we  have $\abs{A_N} > \frac{\abs{S}}{2}$ which by Lemma \ref{S-size} is $>cN$ for all large $N$ and some $c>0$.  
We first bound the summand $\frac{ \abs{ \{a\in S \mid u(u+v) \mid 2a\}}}{ \abs{S}}$.

\begin{lemma}\label{inner-sum-bound} For positive integers $u,v$ with $(u,v)\neq (1,1)$, if some $p \in (2,L)$ divides $u(u+v)$, or $4$ divides $u(u+v)$, or $u(u+v)>2N$, we have
\[ \abs{ \{a\in S\mid u(u+v) \mid 2a\}}=0,\]
and otherwise, we have
\[ \abs{ \{a\in S \mid u(u+v) \mid 2a\}}\]\[ \ll  4^{- \Omega(u)} 4^{- \Omega(u+v)} 4^{ (1+\epsilon) \log \log (u+v)}   \frac{N}{u (u+v)}  \left( (\log (u+v))^{-\frac{3}{4}} (\log L)^{- \frac{1}{4}}  +\bigl( \log \frac{3N}{u (u+v)}\bigr)^{- \frac{3}{4}}\right).\]
\end{lemma}

The base $4$ of the exponentials is not arbitrary. Different values could be chosen, leading to different exponents of the log terms, but this one leads to the best power of log in the final estimate, where it is crucial that a certain power of log is negative. 

\begin{proof} If some $p \in (2,L)$ divides $u(u+v)$ or $4$ divides $u(u+v)$ then it is not possible for $u(u+v)$ to divide $2a$ for $a$ coprime to all primes $p<L$ and hence the sum vanishes. Otherwise \[\abs{ \{a\in S \mid u(u+v) \mid 2a\}} \] \[=  \sum_{ \substack{  d \in [ \delta \frac{2N}{u(u+v)}, \frac{2N}{u(u+v)}] \\  p\nmid d \textrm{ for all } p\in (2,L) , u(u+v)d \equiv 2\bmod 4 \\  \Omega( u,x) + \Omega(u+v, x) + \Omega(d,x) -1\leq (1+\epsilon) \log \log x \textrm{ for all } x}}1 \]
\[\leq\sum_{ \substack{  d \in [ \delta \frac{2N}{u(u+v)}, \frac{2N}{u(u+v)}] \\  p\nmid d \textrm{ for all } p\in (2,L) \\  \Omega( u) + \Omega(u+v) + \Omega(d,u+v) -1\leq (1+\epsilon) \log \log (u+v)}} 1 \]
\[ \leq 4^{- \Omega(u)} 4^{- \Omega(u+v)} 4^{ (1+\epsilon) \log \log (u+v)+1} \sum_{ \substack{  d \in [ \delta \frac{2N}{u(u+v)}, \frac{2N}{u(u+v)} ] \\ p\nmid d \textrm{ for all } p\in (2,L)}} 4^{- \Omega(d,u+v)} \]

The sum $\sum_{ \substack{  d \in [ \delta \frac{2N}{u(u+v)}, \frac{2N}{u(u+v)} ] \\ p\nmid d \textrm{ for all } p\in (2,L)}} 4^{- \Omega(d,u+v)}$ is the sum of a multiplicative function which is $4^{- \Omega(d,u+v)} $ if $p\nmid c$ for $p\in (2,L) $ and $0$ otherwise. This multiplicative function is certainly nonnegative and $1$-bounded. Hence by \cite[Theorem III.3.5]{Tenenbaum2015} we have 

\[ \sum_{ \substack{  d \in [ \delta \frac{2N}{u(u+v)}, \frac{2N}{u(u+v)} ] \\ p\nmid d \textrm{ for all } p \in (2,L) }} 4^{- \Omega(d,u+v)}   \ll  \frac{2N}{u(u+v)} \prod_{p \leq \frac{2N}{u (u+v)} }\begin{cases} 1-p^{-1} & \textrm{if } p\in (2,L)  \\  \frac{1}{1 - \frac{1}{4} p^{-1}} - \frac{1}{p} & \textrm{if } p=2 \textrm{ or }  L \leq p \leq (u+v) \\ \frac{1}{1-p^{-1}} - p^{-1}  & \textrm{if } p>  (u+v) \end{cases}\]
\[ \ll \frac{N}{u(u+v)} \prod_{p \leq \frac{2N}{u (u+v)} , p \leq u+v }\begin{cases} 1-p^{-1} & \textrm{if } p<L \\  1 - \frac{3}{4} p^{-1}  &\textrm{if } L \leq p \end{cases} \]

We have $ \prod_{p \leq \frac{2N}{u (u+v)} , p \leq  (u+v)}(1-\frac{3}{4} p^{-1}) \ll  (\log \min (u+v ,  \frac{3N}{u (u+v)}))^{ -\frac{3}{4}}$ (where we have increased $2$ to $3$ to avoid division by $0$ if $2N=u(u+v)$) and multiplying by $ \prod_{p \leq \frac{2N}{u (u+v)} , p \leq u+v, p <L} \frac{1-p^{-1}}{1-\frac{3}{4}p^{-1}}$ gives an additional factor of $\ll (\log \min  ( u+v, \frac{3N}{u (u+v)}, L))^{-1/4}$ so 
\[\prod_{p \leq \frac{2N}{u (u+v)} , p \leq u+v }\begin{cases} 1-p^{-1} & \textrm{if } p<L \\  1 - \frac{3}{4} p^{-1} & \textrm{if } L \leq p \end{cases}  \]
\[ \ll   \Bigl( (\log (u+v))^{-\frac{3}{4}}+ \bigl(\log \frac{3N}{u (u+v)}\bigr)^{- \frac{3}{4}}\Bigr)  \Bigl( (\log L)^{-\frac{1}{4}} + (\log (u+v))^{-\frac{1}{4}} + \bigl(\log \frac{3N}{u (u+v)}\bigr)^{- \frac{1}{4}} \Bigr) \]
\[ \ll  (\log (u+v))^{-\frac{3}{4}} (\log L)^{- \frac{1}{4}}  +\bigl( \log \frac{3N}{u (u+v)}\bigr)^{- \frac{3}{4}}\]
since we must have $ u+v>2$ so  $p\nmid u+v$ for $p \in (2,L)$ and $ 4\nmid u+v$ implies $u+v \geq L$, so if $u+v \leq \frac{2N}{u (u+v)}$ then the $(\log (u+v))^{-\frac{3}{4}} (\log L)^{- \frac{1}{4}} $ term bounds everything and if $u+v \geq \frac{2N}{u (u+v)}$ then the $(\log \frac{3N}{u (u+v)})^{- \frac{3}{4}}$ term bounds everything. \end{proof}

 Consider the three multiplicative functions 
\[ g_1(n) =\begin{cases} 2^{-\Omega(n)} & \textrm{if }p \nmid n \textrm{ for } p \in (2,L) \textrm{ and } 4\nmid n \\ 0 & \textrm{otherwise} \end{cases}\]
\[ g_2(n) = 2^{- \Omega(n)} \]
\[g_3(n) =\begin{cases} 4^{-\Omega(n)} & \textrm{if }p \nmid n \textrm{ for } p \in (2,L) \textrm{ and } 4 \nmid n  \\ 0 & \textrm{otherwise} \end{cases}\]
and the smooth function \[\mathcal S(u,v) =4^{ (1+\epsilon) \log \log (u+v)}  \frac{1}{u(u+v)} \left( (\log (u+v))^{-\frac{3}{4}} (\log L)^{- \frac{1}{4}}  + \Bigl(\log \frac{3N}{u (u+v)}\Bigr)^{- \frac{3}{4}}\right) .\]
Then we can rephrase Lemma \ref{inner-sum-bound} as 
\begin{equation}\label{inner-sum-eq} \abs{ \{a\in S \mid u(u+v) \mid 2a\}} \ll    2^{\Omega(v)-\Omega(u)} g_1(u) g_2(v) g_3(u+v) N\mathcal S(u,v).\end{equation}

The crucial step is the bound for a sum of $g_1(u) g_2(v) g_3(u+v) $ over dyadic intervals.

\begin{lemma}\label{triple-multiplicative-bound}  For $ X \geq \frac{L}{2(1+\delta^{-1}}$ and, separately, $X$ sufficiently large, we have 
\[ \sum_{\substack{ u \in [X,2X] \\ 0<v \leq u \delta^{-1} \\ (u,v) \textrm{coprime}}} g_1(u) g_2(v) g_3(u+v)  \ll   X^2  (\log X)^{ -\frac{7}{4}} (\log L)^{-\frac{3}{4} }.\]
\end{lemma} 
\begin{proof}  This follows from \cite[Theorem 3.1]{LBT}. To apply \cite[Theorem 3.1]{LBT}, we must introduce notation and parameters, which we now explain.
Take $k=3, t=2$, and primitive polynomials $Q_1=X_1, Q_2=X_2, Q_3=X_1+X_2$. Then, in the notation of \cite[\S2]{LBT}, $Q =Q_1Q_2Q_3=  X_1X_2(X_1+X_2)$. We have $r=3$, $R_1= X_1, R_2 =X_2, R_3=X_1+X_2$ since the $R_h$ are defined as the irreducible factors of $Q$, and we have $Q_j = \prod_h R_h^{\gamma_{jh}}$ for $\gamma_{jh}=\delta_{jh}$. Because the matrix $\gamma_{jh}$ is the identity, we have $\hat{F}=F$. We have $g=\deg Q =3$. For prime $p$ we have \[ \rho_Q^+ (p) =  \sum_{\substack{ u, v\in [1,\dots p]^2\\ u v(u+v)\equiv 0 \bmod p}} 1=3p-2 .\] We let $\kappa(s)$ denote the product of primes dividing $s$ and let $\mathcal K(s_1,s_2,s_3)= \operatorname{lcm}(s_1\kappa(s_1),s_2\kappa(s_2),s_3\kappa(s_3))$. Then we have \[ \rho_{\mathbf R}^\# (s_1,s_2,s_3) =  \sum_{ \substack{ u,v\in [1,\dots, \mathcal K(s_1,s_2,s_3)]\\ s_1 \mid u, s_2 \mid v, s_3 \mid u+v \\  (u/s_1, s_1s_2s_3) = (v/s_2, s_1s_2s_3)= ((u+v)/s_3,s_1s_2s_3)=1 } }1 .\] We take arbitrary $\alpha \in (0,1), \beta \in (0,1)$, $A=1$, $B=1$, and $\epsilon>0$ sufficiently small (note that the $\epsilon$ fixed elsewhere in the paper is different from the $\epsilon$ of \cite[Theorem 3.1]{LBT}). Then 

\[ F(a_1,a_2,a_3) = \begin{cases} g_1(a_1) g_2(a_2)g_3(a_3) & a_1,a_2 \textrm{ coprime} \\ 0 & \textrm{otherwise} \end{cases}\] satisfies $ F( \mathbf a \mathbf b) \leq F(\mathbf a) F(\mathbf b) \leq F(\mathbf b)$ and thus satisfies the bound \[F(\mathbf a \mathbf  b) \leq \min \{ A ^{\Omega(a_1a_2a_3)},B (a_1a_2a_3)^\epsilon\} F(\mathbf b)\] defining $\mathcal M_k(A,B,\epsilon)$. We take $x_1=y_1=2X, x_2=y_2= 2 (1+\delta^{-1}) X$ so that the range of summation $[x_1-y_1,x_1]\times [x_2-y_2,x_2]$ of \cite[Theorem 3.1]{LBT} includes our range of summation. We have $x=\min x_i = 2X$. Then for $X$ sufficiently large, $x_1,x_2,y_1,y_2$ satisfy the hypotheses of \cite[Theorem 3.1]{LBT}. By definition and multiplicativity of $F$ we have
\[ E_{\mathbf R}( x_1+x_2) = E_{\mathbf R} (2 (2+\delta^{-1}) X) = \sum_{ \substack{ s_1,s_2,s_3>0 \\ s_1s_2 s_3 \leq 2 (2+\delta^{-1} )X}} F(s_1,s_2,s_3) \frac{ \rho_{\mathbf R}^\# (s_1,s_2,s_3)}{\mathcal K(s_1,s_2,s_3)^2} \]
\[ \leq \prod_{p \leq 2 (2+\delta^{-1}) X} \sum_{ e_1,e_2,e_3\geq 0} F(p^{e_1} , p^{e_2}, p^{e_3}) \frac{ \rho_{\mathbf R}^\# (p^{e_1},p^{e_2},p^{e_3} )}{ (p^{\max(e_1,e_2,e_3)+1 - 1_{e_1=e_2=e_3=0}})^2}.\]

The term  $F(p^{e_1} , p^{e_2}, p^{e_3})$ always vanishes if $\min(e_1,e_2)>0$ and thus in particular if $\min(e_1,e_2,e_3)>0$.  Thus, when calculating $\rho_{\mathbf R}^\#$, we restrict attention to the case $\min(e_1,e_2,e_3)=0$.

We have $\frac{ \rho_{\mathbf R}^\# (p^{e_1},p^{e_2},p^{e_3} )}{ (p^{\max(e_1,e_2,e_3)+1 - 1_{e_1=e_2=e_3=0}})^2} =1$ if $e_1=e_2=e_3=0$ and otherwise (as long as $\min(e_1,e_2,e_3)=0$)
\[ \frac{ \rho_{\mathbf R}^\# (p^{e_1},p^{e_2},p^{e_3} )}{ (p^{\max(e_1,e_2,e_3)+1 - 1_{e_1=e_2=e_3=0}})^2 } =  \frac{1}{ (p^{ \max(e_1,e_2,e_3)+1})^2}\sum_{ \substack{ u,v\in [1,\dots, p^{\max(e_1,e_2,e_3)+1} ]\\ v_p(u)=e_1, v_p(v)=e_2, v_p(u+v)=e_3}} 1\]
\[ = \begin{cases} 0 & \textrm{if two of }e_1,e_2,e_3\textrm{ are greater than } 0 \\  \frac{(1-p^{-1})^2 }{ p^{\max(e_1,e_2,e_3)} }& \textrm{if  one of }e_1,e_2,e_3\textrm{ is greater than } 0 \end{cases} .\]
For $p\geq L$ we have \[F(p^{e_1}, p^{e_2}, p^{e_3}) = \begin{cases} 2^{-e_1} 2^{-e_2} 4^{-e_3} & \textrm{if }\min (e_1,e_2)=0 \\ 0 & \textrm{otherwise} \end{cases} \] so that we have 
\[  \sum_{ e_1,e_2,e_3\geq 0} F(p^{e_1} , p^{e_2}, p^{e_3}) \frac{ \rho_{\mathbf R}^\# (p^{e_1},p^{e_2},p^{e_3} )}{ (p^{\max(e_1,e_2,e_3)+1 - 1_{e_1=e_2=e_3=0}})^2}\]\[=
1 + \frac{(1-p^{-1})^2}{2p-1}  + \frac{(1-p^{-1})^2 }{2p-1} + \frac{(1-p^{-1})^2 }{4p-1}  = 1 + \frac{5}{4p} + O( \frac{1}{p^2} )\]
where we repeatedly use the geometric series evaluation $\sum_{e=1}^\infty r^{-e} =\frac{1}{r-1}$. 

For $p\in (2,L)$ we have $F(p^{e_1},p^{e_2},p^{e_3})=0$ unless $e_1=e_3=0$ and equals $2^{-e_2}$ in that case, so that
\[  \sum_{ e_1,e_2,e_3\geq 0} F(p^{e_1} , p^{e_2}, p^{e_3}) \frac{ \rho_{\mathbf R}^\# (p^{e_1},p^{e_2},p^{e_3} )}{ (p^{\max(e_1,e_2,e_3)+1 - 1_{e_1=e_2=e_3=0}})^2}=1 + \frac{(1-p^{-1})^2}{ 2p-1} = 1+ \frac{1}{2p} + O(\frac{1}{p^2})\]
and for $p=2$ we have $F(p^{e_1}, p^{e_2} ,p^{e_3})=0$ unless $e_1,e_3\leq 1$ and $\min(e_1,e_2)= 0$ and equals $2^{-e_1} 2^{-e_2} 4^{-e_3}$ when those conditions are both satisfied  so that
\[  \sum_{ e_1,e_2,e_3\geq 0} F(p^{e_1} , p^{e_2}, p^{e_3}) \frac{ \rho_{\mathbf R}^\# (p^{e_1},p^{e_2},p^{e_3} )}{ (p^{\max(e_1,e_2,e_3)+1 - 1_{e_1=e_2=e_3=0}})^2} \ll 1.\]
Thus
\[ E_{\mathbf R}( x_1x_2) \ll \prod_{p<L} (1 +  \frac{1}{2}p^{-1})  \prod_{L \leq p \leq 2 (2+\delta^{-1}) X} (1+ \frac{5}{4} p^{-1}) \ll  (\log X)^{\frac{5}{4}} (\log L)^{- \frac{3}{4}}\]
as long as $X \geq \frac{L}{2 (1+\delta^{-1})}$, which implies $L \leq 2 (2+\delta^{-1}) X$. 
 
 Finally we have \cite[Theorem 3.1]{LBT} which gives
 \[\sum_{ \substack{ u \in [X,2X] \\ 0<v \leq u \delta^{-1} \\ (u,v) \textrm{coprime}}} g_1(u) g_2(v) g_3(u+v)  \]\[\leq \sum_{0 < u  \leq 2X, 0 < v\leq 2(1+\delta^{-1}) X} F(u,v,u+v) \ll X^2E_{\mathbf R} (4 (1+\delta^{-1}) X^2)\prod_{3<p<2X}  \left(1 - \frac{3p-2}{p^2} \right)\]
 \[ \ll  X^2 (\log X)^{\frac{5}{4}} (\log L)^{- \frac{3}{4}} (\log X)^{-3}= X^2  (\log X)^{ -\frac{7}{4}} (\log L)^{-\frac{3}{4} }. \qedhere\]\end{proof}

\begin{lemma}\label{outer-sum-bound} We have 
\[\sum_{\substack{u,v>0 \\ \textrm{coprime}\\ (u,v) \neq (1,1) \\ v \leq u \delta^{-1} \\ \Omega(v) \leq \Omega(u)}} \frac{ \abs{ \{a\in S \mid u(u+v) \mid 2a\}}}{ \abs{S}} < \frac{1}{2}\]
as long as $\epsilon$ is sufficiently small, $L$ is sufficiently large with respect to $\epsilon$ and $\delta$, and $N$ is sufficiently large with respect to $L$. \end{lemma}

\begin{proof} 
From Lemma \ref{inner-sum-bound} in the form \eqref{inner-sum-eq} and Lemma \ref{S-size} we get 
\[\sum_{\substack{u,v>0 \\ \textrm{coprime}\\ (u,v) \neq (1,1) \\ v \leq u \delta^{-1}\\ \Omega(v) \leq \Omega(u)}} \frac{ \abs{ \{a\in S \mid u(u+v) \mid 2a\}}}{ \abs{S}} \]
\[ \ll  \sum_{\substack{u,v>0 \\ \textrm{coprime}\\ (u,v) \neq (1,1) \\ v \leq u \delta^{-1}, u (u+ v)\leq 2N \\ \Omega(v) \leq \Omega(u)  }}  2^{\Omega(v) - \Omega(u)}  g_1(u) g_2(v) g_3(u+v)  \mathcal S(u,v)\frac{2}{   (1-\delta)  \prod_{p<L} (1-p^{-1})}  \]
\[\ll  \sum_{\substack{u,v>0 \\ \textrm{coprime}\\ (u,v) \neq (1,1) \\ v \leq u \delta^{-1} , u (u+ v)\leq 2N\\ \Omega(v) \leq \Omega(u)  }}  2^{\Omega(v) - \Omega(u)}  g_1(u) g_2(v) g_3(u+v) \mathcal S(u,v) \log L  \]
\[ \leq \sum_{\substack{u,v>0 \\ \textrm{coprime}\\ (u,v) \neq (1,1) \\ v \leq u \delta^{-1}, u (u+ v)\leq 2N }}   g_1(u) g_2(v) g_3(u+v)  \mathcal S(u,v) \log L  \]

If we restrict $u$ to a dyadic interval $[X,2X]$, so that $v \leq 2X\delta^{-1}$ and $u+v \in [X,2(1+\delta^{-1}) X]$,  then, since $\frac{3N}{u(u+v)}$ is bounded away from $1$ by the restriction $u(u+v)\leq 2N$, we have 
\begin{equation}\label{smooth-maximum} \mathcal S(u,v) \ll  (\log X)^{(1+\epsilon)\log 4} \frac{1}{X^2} \left( (\log X)^{-\frac{3}{4}} (\log L)^{- \frac{1}{4}}  + \Bigl(\log \frac{3N}{X^2}\Bigr)^{- \frac{3}{4}}\right). \end{equation}
 This sum vanishes if $X < \frac{L}{2 (1+\delta^{-1})}$ as in that case we have $u+v>2$ and $u+v < 2 (1+\delta^{-1}) X<L$ so $u+v$ is necessarily divisible by some odd $p<L$ or by $4$. So we may assume $X\geq\frac{L}{2 (1+\delta^{-1})}$ and thus by taking $L$ sufficiently large assume that $X$ is sufficiently large.  Multiplying the bound of Lemma \ref{triple-multiplicative-bound} by the maximum value \eqref{smooth-maximum} of $\mathcal S(u,v)$ and by $\log L$, we get
\[ \sum_{\substack{u,v>0 \\ \textrm{coprime}\\ (u,v) \neq (1,1), u (u+v) \leq 2N \\ v \leq u \delta^{-1} \\ u \in [X,2X] }}   g_1(u) g_2(v) g_3(u+v)  \mathcal S(u,v) \log L\]
\[ \ll X^2  (\log X)^{ -\frac{7}{4}} (\log L)^{-\frac{3}{4} }(\log X)^{(1+\epsilon)\log 4} \frac{1}{X^2} \left( (\log X)^{-\frac{3}{4}} (\log L)^{- \frac{1}{4}}  + \log (\frac{3N}{X^2})^{- \frac{3}{4}}\right) \log L\]
\[ \ll (\log X)^{(1+\epsilon) \log 4 -\frac{5}{2}}  +( \log  \frac{3N}{X^2})^{- \frac{3}{4}}  (\log X)^{(1+\epsilon) \log 4 -\frac{7}{4}}  (\log L)^{\frac{1}{4} }.\]
We sum with $X$ ranging over all powers of $2$ between $\frac{L}{2 (1+\delta^{-1})}$ and $\sqrt{2N}$. The key fact is that $\log 4 < \frac{3}{2}$ so that for $\epsilon$ sufficiently small we have $(1+\epsilon) \log 4 -\frac{5}{2}<-1$.

By the key fact, for the first term, the sum over powers of $2$ is convergent and thus is arbitrarily small for $L$ sufficiently large.

For the second term, when $\log X< (\log N)/4$ we have $\log N \ll \log (\frac{3N}{X^2})$ so the second term is bounded by   $(\log N)^{- \frac{3}{4}}  (\log X)^{(1+\epsilon) \log 4 -\frac{7}{4}}  (\log L)^{\frac{1}{4} }$. The exponent of $\log X$ is greater than $-1$ so when we sum this term over $X$ powers of $2$ with $\log X < (\log N)/4$ we obtain  $( \log N)^{- \frac{3}{4} + (1+\epsilon) \log 4 -\frac{7}{4}+1 }  (\log L)^{\frac{1}{4} }$. The exponent of $\log N$ is negative by the same key fact, so this can be made arbitrarily small by taking $N$ sufficiently large with respect to $L$. 

When $\log X> \log N/4$ we have $\log N \ll \log X$ and so the second term is \[\ll (\log \frac{3N}{X^2})^{- \frac{3}{4}}( \log N )^{(1+\epsilon) \log 4 -\frac{7}{4}}  (\log L)^{\frac{1}{4} }.\] The exponent of $(\log \frac{3N}{X^2})$ is greater than $-1$ so when we sum this term over the possible values of $X$ we obtain is $ (\log N) ^{-\frac{3}{4}+ 1+ (1+\epsilon) \log 4 -\frac{7}{4}}  (\log L)^{\frac{1}{4} }$ which by the same key fact has a negative power of $\log N$ and thus can be made arbitrarily small. \end{proof}

\begin{proof}[Proof of Theorem \ref{main}] (1) is clear since of any pair $a,b$, we must have $\Omega(a) \leq \Omega(b)$ or $\Omega(b) \leq \Omega(a)$, so if $a+b \mid 2ab$ it not possible for both $a$ and $b$ to be in $A_N$.

 By Lemma \ref{in-A-criterion}, any element $a\in S$ with $a\notin A_N$ must have $2a$ divisible by $u(u+v)$ for two coprime positive integers $u,v$ with $(u,v)\neq (1,1)$ and $v\leq u \delta^{-1}$ and $\Omega(v) \leq \Omega(u)$. By Lemma \ref{outer-sum-bound} it follows that $\frac{\abs{ S \setminus  A_N}}{\abs{S}}< \frac{1}{2}$ and thus that $\abs{A_N}> \frac{\abs{S}}{2}$ which by Lemma \ref{S-size} gives (2).\end{proof}

\bibliographystyle{plainurl}

\bibliography{references}

\end{document}